\date{}
\title{Indistinguishability of the components of random spanning forests}

\author{ \'Ad\'am Tim\'ar\\ 
\small Alfr\'ed R\'enyi Institute of Mathematics\\[-0.8ex]
\small Re\'altanuda u. 13-15.,\\
\small H-1053 Budapest\\
\small \texttt{madaramit[at]gmail.com}\\}

\documentclass[10pt,naturalnames]{article}
\renewcommand\footnotemark{}
\usepackage[ruled,vlined]{algorithm2e}
\usepackage{amsmath}
\usepackage{amsthm}
\usepackage{amsfonts}
\usepackage{graphicx}
\newif\ifhyper\IfFileExists{hyperref.sty}{\hypertrue}{\hyperfalse}
\ifhyper\usepackage{hyperref}\fi
\usepackage{enumitem}
\usepackage{subfig}
\usepackage{caption}
\usepackage{keyval}
\usepackage[margin=1in]{geometry}
\usepackage{setspace}
\usepackage[displaymath, mathlines,pagewise]{lineno}

\newif\ifdraft
\drafttrue
\numberwithin{equation}{section}
\numberwithin{figure}{section}

\theoremstyle{definition}
\newtheorem{exmp}{Example}[section]
\newtheorem{theorem}{Theorem}
\numberwithin{theorem}{section}
\newtheorem{corollary}[theorem]{Corollary}
\newtheorem{lemma}[theorem]{Lemma}

\newtheorem{remark}[theorem]{Remark}
\newtheorem{proposition}[theorem]{Proposition}

\newtheorem{definition}{Definition}

\newcommand{\Z}{\mathbb{Z}}

\def \P {{\bf P}}

\def \E {{\bf E}}

\def \_reg {\rightarrow_{\bf reg}}

\def\maxdeg/{\Delta}

\def\max{{\rm max}}
\def\dist{{\rm dist}}

\def\dist{{\rm dist}}
\def\W{{\cal W}}
\def\F{{\cal F}}
\def\T{{\cal T}}
\def\B{{\cal B}}
\def\A{{\cal A}}
\def\Pef{{\cal P}_{e,x}^f}
\def\Pefn{{\cal P}_{e, x}^{f}}

\def\piin{{\pi_{e}^{f}}}
\def\pii{\pi_e^f}
\def\Ene{{\cal E}^n_{x,e,f}}
\def\phie{\phi (e,\lambda)}
\def\ff{f(\omega,e,x,r)}

\begin{document}
\maketitle
\let\thefootnote\relax\footnotetext{\footnotesize{\it{Keywords and phrases:} Spanning Forests, Uniform Spanning Forest, Minimal Spanning Forest, insertion tolerance, indistinguishability


The author is a Marie Curie fellow.}}

\bigskip

\begin{abstract}
We prove that the infinite components of the Free Uniform Spanning Forest (FUSF) of a Cayley graph are indistinguishable by any invariant property, given that the forest is different from its wired counterpart. Similar result is obtained for the Free Minimal Spanning Forest (FMSF). We also show that with the above assumptions there can only be 0, 1 or infinitely many components. These answer questions by Benjamini, Lyons, Peres and Schramm for the case of unimodular random graphs, which have been open for the FUSF up to now. Our methods apply to a more general class of percolations, those satisfying ``weak insertion tolerance", and work beyond Cayley graphs, in the more general setting of unimodular random graphs.


\bigskip
\bigskip
\end{abstract}

\section{Introduction}\label{s.intro}

We prove indistinguishability and 1-infinity laws for the components ({\it clusters}) of random spanning forests of Cayley graphs, given that the forest has a property that we call weak insertion tolerance (see Definition \ref{maindef}), and it has a tree with infinitely many ends. The perhaps most important examples of random forests that satisfy weak insertion tolerance are the Free and the Wired Uniform Spanning Forest (FUSF and WUSF) and the Free and the Wired Minimal Spanning Forest (FMSF and WMSF). See Definitions \ref{usf}, \ref{msf}. The importance and some main properties of the uniform and minimal spanning forests are explained in \cite{LP}.

In particular, the following theorems are proved.

\begin{theorem}\label{USF_all}
Suppose that the FUSF and WUSF are different for some unimodular quasitransitive graph $G$. Then the following hold:
\begin{enumerate}
\item The FUSF has either 1 or infinitely many components.
\item Every component of the FUSF has infinitely many ends.
\item More generally, no two components of the FUSF can be distinguished by any invariantly defined property.
\end{enumerate}
\end{theorem}

The condition FUSF$\not=$WUSF is equivalent to that there exist nonconstant harmonic Dirichlet functions on $G$, or, in different terms, that the first ${\rm L}^2$ Betti number is nonzero. This was shown by Benjamini, Lyons, Peres and Schramm, see \cite{BLPS}.

\begin{theorem}\label{MSF_all}
Suppose that the FMSF and WMSF are different for some unimodular quasitransitive graph $G$. Then the following hold:
\begin{enumerate}
\item The FMSF has either 1 or infinitely many components.
\item Every component of the FMSF has infinitely many ends.
\item More generally, no two components of the FMSF can be distinguished by any invariantly defined property.
\end{enumerate}
\end{theorem}

The condition FMSF$\not=$WMSF is equivalent to $p_c<p_u$, as shown by Lyons, Peres and Schramm. Here $p_c$ and $p_u$ are respectively the critical probability and uniqueness critical probability for Bernoulli percolation on $G$. The condition $p_c<p_u$ is conjecturally equivalent to $G$ being nonamenable, and is known to hold for {\it some} Cayley graph of every nonamenable group. See \cite{LPS} for more details.

All the results of this paper, including Theorems \ref{USF_all} and \ref{MSF_all}, remain valid if $G$ is a {\it unimodular random graph}. See \cite{AL} for the definition of this notion, which includes all unimodular quasitransitive graphs (or more generally,  invariant random subgraphs of a unimodular quasitransitive graph). We present the proofs for unimodular quasitransitive graphs because this setting is more widely known. Remark \ref{unimod} describes the extra details needed for the proofs to be applied to a unimodular random $G$.

The above theorems follow from Lemma \ref{lemmamain1}, Corollary \ref{instolMSF} and Theorem \ref{comps} (Part 1), Theorem \ref{ends} (Part 2), Theorem \ref{indist} (Part 3). One needs that the uniform and the minimal spanning forests are ergodic, which were proved in \cite{BLPS} and \cite{LPS} respectively. What needs to be further added is that
FUSF$\not =$WUSF implies that some tree of FUSF has infinitely many ends; and similarly for the minimal spanning forest. For the uniform spanning forest this is true by  Proposition 10.11 in \cite{BLPS}  and for the minimal spanning forest this is part (e) of Proposition 3.5 in \cite{LPS}.

Theorems \ref{USF_all} and \ref{MSF_all} resolve questions asked by Benjamini, Lyons, Peres and Schramm \cite{BLPS} and by Lyons, Peres and Schramm \cite{LPS}. 
Part 1 in Theorem \ref{USF_all} answers Question 15.6 in \cite{BLPS}, Part 2 answers Question 15.8 for the case when the transitive graph is unimodular, while Part 3 confirms Conjecture 15.9 in the same paper for the case of FUSF when FUSF$\not=$ WUSF. Part 1 of Theorem \ref{MSF_all} was  Question 6.7 in \cite{LPS} and was answered in \cite{T_msf} using a different method as here. 
Parts 2 and 3 answer Question 6.10 and Conjecture 6.11, respectively, for the case of FMSF when FMSF$\not=$WMSF. These were first solved by Chifan and Ioana in \cite{CI} (see Corollary 9), using operator algebraic techniques and the result of \cite{T_msf} that the number of ends is the same in every component. Our approach provides a direct probabilistic proof of the same theorem.
The conjecture on the indistinguishability of FMSF-clusters was restated by Gaboriau and Lyons in \cite{GL}, because in case of a positive answer (as provided by Theorem \ref{MSF_all}), the FMSF can serve as the treeable ergodic subrelation in their construction (Proposition 13 of \cite{GL}) for some Cayley graph of the given group. After finishing the first draft of this manuscript, we learnt that Hutchcroft and Nachmias gave an independent proof, about the same time, to the results of Theorem \ref{USF_all}, assuming only transitivity \cite{HN}. Their paper further shows the same conclusions for the WUSF.

Let $G$ be the underlying unimodular quasitransitive graph (such as a Cayley graph), with vertex set $V(G)=V$ and edge set $E(G)=E$. Denote by $d$ the (maximal) degree in $G$. Let $\dist (u,v)$ be the distance between $u$ and $v$ in $G$, where $u,v\in V\cup E$.
Denote by $B(x,r)$ the ball of radius $r$ around $x$ in $G$, and by $B_\Gamma (x,r)$ the ball of radius $r$ around $x$ in some given subgraph $\Gamma$ of $G$. Let $S(x,r)$ and $S_\Gamma (x,r)$ be the spheres of radius $r$ around $x$ in $G$ and in $\Gamma$ respectively.
Given some percolation (random subgraph) on $G$, the component of a given vertex $x$ will be denoted by $C_x$. (We hide the dependence of $C_x$ on $\omega$ for the ease of notation.) Given $e\in E$, $f\in E\cup\{\emptyset\}$ and a configuration $\omega\in 2^E$, let $\pii\omega:=\omega\cup e\setminus f$. For
an event $A$, let $\pii A:=\{\omega\cup e\setminus f\,:\, \omega\in A\}$. Denote the complement of a set $A$ within some superset (clear from the context) by $A^c$.
We will use $\P$ for different probability measures in the paper, but its meaning will be always clear from the context. We will use $\E$ for the expectation of $\P$.

A percolation is called insertion tolerant (see \cite{LS}), if one can insert a fixed edge to each configuration of a given event $A$ and obtain an event of positive probability after the insertion, provided that the original event had positive probability. The key property needed for our proofs is a weak form of insertion tolerance, as given in the next Definition. Informally, this is the following modification of the ``usual" notion of insertion tolerance. First, one may assume that the event $A$ is such that the endpoints of the edge $e$ are in distinct components on $A$. (This is not a real constraint, since in applications one usually wants to insert $e$ if it is between two components.) Then we can insert $e$ to the configurations in $A$, but at the cost of possibly deleting another edge $f$. Furthermore, this $f$ can be chosen for any fixed $r\geq 0$ to be at distance at least $r$ from $e$, and it can be chosen so that it is in the component of a previously fixed endpoint $x$ of $e$.

\begin{definition} \label{maindef} ({\bf weak insertion tolerance, WIT})
We say that a random forest $\F$ of a unimodular quasitransitive graph $G$ is weakly insertion tolerant (WIT) if 
for any $\{x,y\}=e\in E(G)$, $r$ nonnegative integer,  and configuration $\omega$ such that $x$ and $y$ are in different components,
there exists an $f=f(\omega,e,x,r) \in E(G)\cup\{\emptyset\}$ such that the following holds. Fixing $e,x,r$ and looking at $f$ as a function of $\omega$, it is measurable.
If $A$ is such that $\P (A)>0$ and for almost every configuration in $A$, $C_x\not = C_y$, then   $\P (\pii A)>0$. Furthermore, if $f\not=\emptyset$ then $f$ is in $C_x\cap B(x,r)^c$ almost surely.
\end{definition}

Suppose that $G$ is an infinite unimodular graph and $G_n\subset G$ is an exhausting sequence of connected finite graphs. Let $\hat G_n$ be obtained from $G_n$ by adding an extra vertex $z_n$ to it, and replacing every edge of the form $\{x,y\}\in E(G)$, $x\in V(G_n)$, $y\not\in V(G_n)$, by a copy of the edge $\{x,z_n\}$.

\begin{definition} \label{usf} ({\bf Uniform Spanning Forest})
Let $G$ be an infinite graph and $G_n\subset G$ be an exhausting sequence of connected finite graphs. Let $T_n$ be a uniformly chosen spanning tree of $G_n$, and $\hat T_n$ be a uniformly chosen spanning tree of $\hat G_n$. Pemantle showed that the weak limits of $T_n$ and of $\hat T_n$ exists \cite{P}. The first one is called the Free Uniform Spanning Forest (FUSF) of $G$, the second one is called the Wired Uniform Spanning Forest (WUSF) of $G$.
\end{definition}

\begin{definition} \label{msf} ({\bf Minimal Spanning Forest})
Let $G$ be an infinite graph and $\lambda$ be an i.i.d. labelling of its edges by Lebesgue[0,1] labels.
Delete each edge from $G$ if its label is maximal in some (finite) cycle of $G$, and call the remaining random forest $F_F(\lambda)$,  
the Free Minimal Spanning Forest (FMSF) of $G$.
Alternatively, delete each edge from $G$ if its label is maximal in some cycle or biinfinite path of $G$, and call the remaining random forest $F_W(\lambda)$, the Wired Minimal Spanning Forest (WMSF) of $G$.
\end{definition}

We mention that the WMSF and FMSF can equivalently be defined using an exhausting sequence of finite graphs, see \cite{LPS} for the details.

\begin{lemma} \label{lemmamain1}
The Free Uniform Spanning Forest and the Wired Uniform Spanning Forest are weakly insertion tolerant. Moreover, there exists a uniform $\delta(r)$ such that for any $A$, $e$ and $f$ as in Definition \ref{maindef}, $\P(\pii A)>\delta (r)\P(A)$. 
\end{lemma}

For a transitive graph $G$ and $p\in [0,1]$, denote by $\theta (p,G)$ the probability that $o$ is in an infinite component of Bernoulli($p$) edge percolation, where $o\in V(G)$ is some fixed vertex. Whether $\theta(p_c,G)=0$, is a central open problem, known to be true when $G$ is nonamenable, \cite{BLPSperc}. We mention that the situation FMSF$\not=$WMSF, which is studied below, can only happen when $G$ is nonamenable, \cite{LPS}.

\begin{lemma} \label{lemmamain2}
Let $G$ be such that $\theta(p_c,G)=0$. 
The Free Minimal Spanning Forest and the Wired Minimal Spanning Forest are weakly insertion tolerant.
\end{lemma}

\begin{corollary} \label{instolMSF}
Suppose FMSF$\not=$WMSF. Then the FMSF and the WMSF are weakly insertion tolerant.
\end{corollary}

A standard tool in the study of percolations on transitive graphs is the so-called Mass Transport Principle (MTP). In brief, it says that if $x$ sends mass $\phi (\omega , x,y)$ to $y$ and this mass transport function is diagonally invariant, then the expected total mass $\E\sum_y \phi (\omega , x,y)$ sent out by $x$ is the same as the expected total mass $\E\sum_y \phi (\omega ,y,x)$ received by $x$. The most typical use of the MTP is that there is no way to assign some vertex to each vertex in an invariant way such that some vertex is assigned to infinitely many other vertices
with positive probability.
See Section 8.1 in \cite{LP} for more details and the history of the MTP.


\section{Uniform and Minimal Spanning Forests are WIT}

The perhaps most important examples of weakly insertion tolerant forests are the Uniform and the Minimal Spanning Forests. For the latter case we are only able to prove WIT if we assume $\theta(p_c,G)=0$, although a weaker version of WIT holds in full generality, namely, if we do not require $\ff$ to be in $C_x$. See Remark \ref{WIT_pocsoles} for an explanation of what benefits and losses there would have been of such an alternative definition of WIT. 

\begin{proofof}{Lemma \ref{lemmamain1}}
Fix $e=\{x,y\}$. Denote a random forest on $G$ by $F$. First consider the case of FUSF.

Fix $r$. Let $G_j$ be an exhaustion of $G$ by finite graphs, and $T_j$ be the FUSF
of $G_j$. We may assume that every $G_j$ contains the ball $B(x,r+1)$ of radius $r$ around $x$. 
Denote the path from $x$ to $y$ in $T_j$ by $P_j$. 
Let $D_j:=\{x$ and $y$ are in different components of $\{B(x,r+1)\cap T_j\}\}$, and let $D:=\{x$ and $y$ are in different components of $\{B(x,r+1)\cap F\}\}$. By the convergence of $T_j$ to $F$, $\P (D_j)$ converges to $\P (A)$. 

Recall that one can define a metric on (rooted equivalence classes of) rooted graphs, where the distance between rooted graphs $(\Gamma_1,o_1)$ and $(\Gamma_2,o_2)$ is $1/r$ if $r$ is the maximal integer such that the $r$-neighborhoods of $o_1$ in $\Gamma_1$ and $o_2$ in $\Gamma_2$ are rooted isomorphic. This metric defines a Polish space. By the Skorokhod representation theorem,
the weak convergence of $T_j$ to $F$ implies the existence of a coupling between $(T_n)_n$ and $F$ such that $T_n$ converges to $F$ a.s. 
Conditioned on $D_j$, $P_j\cap S(x,r+1)\not=\emptyset$. Let $f_j$ be the first edge of $S(x,r+1)$ when going along $P_j$ starting from $x$. Since $T_j$ converges to $F$, its restriction to $B(x,r+1)$ also converges to that of $F$ a.s. In particular, $f_j$ has a limit as $j\to\infty$. Let $f$ be the random edge given by this limiting distribution. Note that $f\in C_x\cap S(x,r+1)$. 

Now assume that $F$ is given by the FUSF. Under the map $\phi_j :T_j\mapsto T_j\cup e\setminus f_j$, every configuration has at most $|S(o,r+1)|$ preimages (at most one for each potential $f_j\in S(o,r+1)$). It follows that for any event $A_j\subset D_j$, $\P(\phi_j(A_j))\geq |S(o,r+1)|^{-1}\P(A_j)$. For any $A\subset D$ one can choose an approximating sequence $A_j\to A$, $A_j\subset D_j$. Hence $\P (\pii A)=\lim_j \P (\phi_j (A_j))\geq |S(o,r+1)|^{-1}\lim _j\P(A_j)=|S(o,r+1)|^{-1}\P(A)$.
This finishes the proof for FUSF.

Similar argument works for the WUSF with $G_j$ replaced by $\hat G_j$.

\end{proofof}
\qed

Given a labelling $\lambda : E(G)\to [0,1]$ and $e\in E(G)$, 
define
$Z_\lambda (e)=Z(e)=
\inf_C \sup \{\lambda (e'):\, e'\in C\setminus \{e\}\}$, where the infimum is over all cycles $C$ in $G$ that contain $e$. Depending on the context, by a cycle we may mean only finite cycles (in case of the FMSF) or finite cycles and biinfinite paths (in case of WMSF). If the infimum of the sup is attained in the definition of $Z(e)$, denote the edge $e'$ by $\phie$ (then $\lambda (\phie)=Z(e)$), otherwise let $\phie=\emptyset$.

In the edge labellings considered below, we assume that all labels are different. This holds with probability 1 when the labels are i.i.d. Lebesgue$[0,1]$. For the proof of Lemma \ref{lemmamain2}, we will need the following observation. 

\begin{proposition}\label{labels}
Let $e\in E(G)$, $\lambda :E(G)\to [0,1]$ be a labelling and $\lambda'$ be another labelling that agrees with $\lambda$ for every edge other than $e$.
\begin{enumerate} 
\item Suppose that $e\in E(G)$ is in $F_F(\lambda)$. If $\lambda '(e)<Z_\lambda (e)$ then $F_F(\lambda)=F_F(\lambda')$. If $\lambda'(e)>Z_\lambda(e)$ then
$F_F(\lambda')=\{\phie\}\cup F_F(\lambda)\setminus \{e\}$.
\item Suppose that $e\in E(G)$ is not in $F_F(\lambda)$. If $\lambda '(e)>Z_\lambda (e)$ 
then $F_F(\lambda)=F_F(\lambda')$.  If $\lambda '(e)< Z_\lambda(e)$ 
$F_F(\lambda')=\{e\}\cup F_F(\lambda)\setminus \{\phie\}$.
\end{enumerate}
Statements similar to (1) and (2) hold with $F_F$ replaced by $F_W$ above (with $Z_\lambda $ changed to the free version).
\end{proposition}

\begin{proof}
The proposition follows from the proof of Lemma 3.15 in \cite{LPS}. There it is shown that $\lambda |_{E\setminus\{e\}}$ determines $F(\lambda) |_{E\setminus\{e,\phi (e,\lambda)\}}$ (note $\phi (e,\lambda)$ is determined by $\lambda |_{E\setminus\{e\}}$), and that one has either $e\in F(\lambda)$, $\phi (e,\lambda )\not\in F(\lambda)$ (iff $\lambda (e)<Z(e)=\lambda (\phi (e,\lambda ))$ or if $\phi (e,\lambda )=\emptyset$), or $e\not\in F(\lambda)$, $\phi (e, \lambda)\in F(\lambda)$ (iff $\lambda (e) > \lambda (\phi (e,\lambda ))$). The Proposition follows from these.
\end{proof}
\qed

\begin{proofof}{Lemma \ref{lemmamain2}}
First we show the claim for the WMSF. In \cite{HPS} it is proved that in case of $\theta(p_C,G)=0$, every component of the WMSF has one end. (In general, WMSF components can have at most two ends.)

Let $\lambda$ be the random labelling and the corresponding spanning forest be $F(\lambda)=F_W(\lambda)$. 
Given $e\in E(G)$, we will use ${\cal C}(e)$ for the set of all cycles containing $e$, where cycles are understood to be finite or biinfinite paths.
Fix $e$, $r$ and $A$ as in Definition \ref{maindef}. Recall the definition of $f(\omega,e,x,r)$, and note that $\omega$ here stands for a realization $F(\lambda)$. Condition on $A$, i.e., suppose that $F(\lambda)\in A$.

To prove the claim, for any labelling $\lambda$ we will define a labelling $\lambda'$ whose properties are explained next. We will have $F(\lambda)=F(\lambda')$.
If we change $\lambda' (e)$ to a $\lambda'' (e)<Z_{\lambda'}(e)$ leaving all other labels unchanged, then Proposition \ref{labels} applies, and hence 
$F(\lambda'')=F(\lambda')\cup e\setminus \phi (e,\lambda')$. Now, $\lambda'$ will have the property that if $\phi (e,\lambda')\not=\emptyset$ then it satisfies the requirements for $f(F(\lambda'),e,x,r)$: its distance from $x$ is at least $r$, and it is in $C_x$.
Furthermore, the map $\lambda\to\lambda'$ will be measurable and it will take sets of positive probability to sets of positive probability. This will prove the Lemma.

Let $Z_1(e):=Z(e)$ and $f_1:=\phie$. Let $i\in \Z^+$ and suppose that $Z_1(e),\ldots , Z_{i-1}(e)$ have been defined and that $i<k$, where $k$ is to be determined later. Then define $Z_i(e):=\inf_{C\in {\cal C}(e)}\sup \{\lambda (e')\,:\, e'\in C\cap E(G)\setminus\{f_1,\ldots, f_{i-1}\}\}$ and let $f_i$ be the edge where this inf sup is attained, if there is any, otherwise let $f_i:=\emptyset$.

Let $k$ be the smallest number such that $f_k\not\in C_y\cup (B(x,r)\cap C_x)$ (including the case when $f_k=\emptyset$). 
If such a $k$ does not exist, define $k$ to be infinity. The set $C_y\cup (B(x,r)\cap C_x)$ does not contain any element of $C(e)$ as a subset, thus
the sup in the definition of $Z_i (e)$ is always taken over some nonempty set as long as $i\leq k$.)

Suppose first that $k$ is finite. Define the following labelling $\lambda'$ from $\lambda$:

(i) $\lambda'(f_i):=\lambda (f_i)Z_k(e)$ for all $i<k$,

(ii) leave all other labels unchanged.

By the choice of $k$, in (i) we only decrease labels of edges in $F(\lambda)$, hence the minimal spanning forest is not changed by these changes of labels, by Proposition \ref{labels}. It is easy to check that the map $\lambda\to\lambda'$ is measurable and it takes sets of positive probability to sets of positive probability. In the new labelling $\lambda'$ we have $\phi (e,\lambda')=f_k$. This implies $f_k\in C_x\cup C_y$, for the following reason. If we decrease the label of $e$ in $\lambda'$ (as in the definition of $\lambda''$ above), so that it becomes part of the forest ($\lambda'(e)<Z_k(e)$), the edge $\phi (e,\lambda')=f_k$ that drops out of the forest is outside of $C_x\cup C_y\cup \{e\}$. Then $C_x\cup C_y\cup\{e\}$ is one of the new components, with two ends, a contradiction.
We conclude using the definition of $k$ that if $k$ is finite then $\phi (e,\lambda')=f_k$ is in $C_x\cap B(x,r)^c$, as we wanted.

Finally, suppose that $k$ is infinity. Then for some $K>0$ we have $f_i\in C_y$ for every $i>K$ (where $K=|B(x,r)|$). 

For $\alpha\in [0,1]$, define $G_\alpha=G_\alpha (\lambda)=\{e\in E(G)\,:\,\lambda (e)<\alpha\}$. Let $i\geq 1$ be arbitrary. Let $O_i$ be a cycle such that $f_i$ is maximal in $O_i\setminus\{e,f_1,\ldots, f_{i-1}\}$. It
has the property that $O_i\setminus\{e,f_1,\ldots, f_{i-1}\}\subset G_{\lambda(f_i)}$ a.s.. By definition of $f_i$, there exists such a cycle with $e\in O_i$. It follows from the choice of 
$O_i$ (also applied to $O_{i-1}$) that
$$\lambda (f_{i})\leq \max \left\{\lambda (f)\,:\, f\in O_i\setminus\left\{e,f_1,\ldots, f_{i-1}\right\}\right\}=\inf_{C\in {\cal C}(e)}\sup \left\{\lambda (e')\,:\, e'\in C\cap E(G)\setminus\{f_1,\ldots, f_{i-1}\}\right\}$$
$$\leq \max \left\{\lambda(f)\,:\,f\in O_{i-1}\setminus \{e,f_1,\ldots,f_{i-1}\}\right\}<\lambda(f_{i-1}).$$ 
Note that for the second inequality we need that $O_{i-1}\setminus\{e,f_1,\ldots,f_{i-1}\}\not=\emptyset$, but this is guaranteed by the fact that $f_1,\ldots,f_{i-1}\in C_x\cup C_y$ and that $|O_{i-1}\setminus(C_x\cup C_y)|\geq 2$.
We obtained that $\lambda$ is monotone decreasing on the sequence $f_i$.
By this monotonicity, if a cycle in $G_{\lambda(f_j)}\cup\{e,f_1,\ldots, f_{j-1}\}$ contains $e$ and $f_j$, then it also contains $f_1,\ldots, f_{j-1}$. Otherwise, if $O$ is such a cycle and $i$ is the smallest index with $f_i\not\in O$, then an element of $\{f_{i+1},\ldots, f_j\}\cap O$ would have been chosen for $f_i$ by definition.

we would contradict the choice of the $f_i$. Next we show that all the $f_i$ ($i>K$) are on the infinite ray from $y$ in $C_y$. We have just observed that if $e,f_i\in O_i$ and $\lambda (f_i)$ is maximal in $O_i\setminus\{e,f_1,\ldots, f_{i-1}\}$, then $e,f_1,\ldots, f_i\in O_i$. By assumption, all the $f_j$ are in $C_y$ for $j>K$. The path $P$ between $f_j$ and $f_{j+1}$(not containing $f_j$ and $f_{j+1}$) in $C_y$ minimizes the max of labels among all such paths, because for any other subpath $P'$ between them, $P\cup P'$ is a cycle, with the edge of maximal label not being in $P\subset F(\lambda)$, hence the max of $\lambda$ over $P\cup P'$ is attained in $P'$. Therefore we may choose $O_i$ such that the subpath of $O_i$ between $f_j$ and $f_{j+1}$ that does not contain $e$ is inside $C_y$ (for any $O_i$, we can replace the subpath between $f_j$ and $f_{j+1}$ by $P$, and $f_i$ is still maximal in $O_i\setminus\{e,f_1,\ldots, f_{i-1}\}$). To summarize: we have seen that for $j\geq i>K$, all the $f_i$ are on a subpath of $O_i$ within $C_y$. Letting $j$ tend to infinity we obtain that all the $f_i$ ($i>K$) are on the infinite ray from $y$ in $C_y$.

We have seen that for $i>K$, $f_i$ has maximal label on the infinite ray from $f_i$ in $C_y$, hence
necessarily $\lambda (f_i)\geq p_c$.
For every $p>p_c$ every component of the WMSF intersects the cluster $G_p(\lambda)$ in an infinite component (Lemma 3.11 in \cite{LPS}). This implies $$\lim \lambda (f_i)=p_c,$$ 
using again that the trees of the WMSF are 1-ended. Take a subsequential limit of the $O_i$, call the resulting biinfinite path $B$. All the $f_i$ are on one side of $e$ in $B$ for $i>K$, hence there is an infinite path $P\subset B$ such that $P\cap\{e,f_1,f_2,\ldots\}=\emptyset$. Take an arbitrary edge $g\in P$. on the other side of $e$. If $\lambda (g)>p_c$, then if $i$ is large enough then $f_i$ we have $\lambda (f_i)<\lambda (g)$. By definition of $B$, if $i$ is large enough, then further $O_i$ contains $g$. But this contradicts the choice of $f_i$, because then $\lambda (f_i)$ is not maximal in $O_i\setminus \{e,f_1,\ldots, f_{i-1}\}$. We conclude that every $g\in P$ has $\lambda (g)\leq p_c$, contradicting $\theta (p_c)=0$. 
This final contradiction shows that
$k$ cannot be infinite, and the proof is finished.

The case of FMSF follows from the previous proof: note that we constructed $\lambda'$ by lowering the labels of some edges in $F_W(\lambda)\subset F_F(\lambda)$. Hence $F_W(\lambda')=F_W(\lambda)\subset F_F(\lambda)=F_F(\lambda')$. When reducing the label of $e$ in $\lambda'$ (below $Z_{{\rm Free}}(e):=\inf_C \sup \{\lambda (e'):\, e'\in C\setminus \{e\}\}$, $C$ finite cycle containing $e$), then $e$ becomes part of the forest, and either no edge drops out of it, or the edge that drops out is $\phi (e,\lambda')$, which satisfies the requirements for $f(F(\lambda'),e,x,r)$.

\end{proofof}
\qed


\begin{remark}\label{WIT_pocsoles}
In the definition of WIT, the requirement that $f\in C_x$ is needed only for the proof of Theorems \ref{ends} and \ref{comps}. Theorem \ref{indist} is true without this assumption, if we know that the conclusion of Theorem \ref{ends} holds. Along the lines of the above proof one could show that WMSF and FMSF are weakly insertion tolerant without the assumption $\theta(p_c,G)=0$, if we had chosen the less restrictive form of weak insertion tolerance, where $\ff$ need not be in $C_x$.

\end{remark}

\section{Number of components, number of ends} \label{s.main}

\begin{theorem} \label{ends}
Let $G$ be a unimodular quasitransitive graph, and ${\cal F}$ an ergodic random spanning forest of $G$. If ${\cal F}$ satisfies weak insertion tolerance and one of its components has infinitely many ends then every component has infinitely many ends.
\end{theorem}

\begin{proof}
Suppose by contradiction that there is also some component with finitely many ends. Then there is a vertex $x$, edge $e=\{x,y\}$ and event $A_0$ with $\P(A_0)>0$ such that conditioned on $A_0$, $x$ is in a component $C$ with infinitely many ends, and $y$ is in a different component $C'$ with finitely many ends. To see this, note that with probability 1 there exist adjacent components such that one of them has infinitely many ends and the other one has finitely many ends. Then for some fixed edge $e$, there is a positive probability that $e$ connects two such components - otherwise, summing up over the countably many edges, we would get an event of 0 probability, contradicting the previous sentence. A similar argument will be used later several times without explicit mention. Namely, if there exists an edge of a certain property with positive probability, then there exists a fixed edge $e$ that has this property with positive probability.

Choose $r>0$ such that $C\setminus B(x,r)$ has at least 3 infinite components with probability at least $\P (A_0)/2$. Such an $r$ exists by the assumption on $C$. Let $A$ be the subevent of $A_0$ that $C\setminus B(x,r)$ has at least 3 infinite components. In particular, $\P(A)\geq \P(A_0)/2>0$. Now consider $f=f(\omega,e,x,r)$ as in the definition of weak insertion tolerance and take $\pii A$. Then for every $\omega\in A$, the component of $x$ in $\omega\cup\{e\}\setminus\{f\}\in \pii A$ contains $C'$, $e$ and at least 2 ends from $C$. 
In particular, it has an isolated end (in $C'$), which is impossible by a standard MTP argument (see Proposition 3.9 in \cite{LP}). By WIT, $\P(\pii A)>0$, giving a contradiction. 
\end{proof}
\qed

The next lemma summarizes some well-known claims that we will need later. 

\begin{lemma}\label{transient}
Consider an invariant edge-percolation process on $G$ whose components are infinite trees. Let $o\in V(G)$.

(i) If $C_o$ has infinitely many ends, then it is transient.

(ii) If the expected degree of $o$ is strictly greater than 2 then some component has infinitely many ends. Conversely, conditioned on that $o$ is in a component with infinitely many ends, its expected degree in the component is greater than 2.

(iii) If $C_o$ has infinitely many ends, then it has exponential growth.

(iv) If $C_o$ has infinitely many ends, then for any finite $S\subset V\cup E$ every infinite component of $C_o\setminus S$ has infinitely many ends.
\end{lemma}

\begin{proof}
The proof of Proposition 3.11 in \cite{LS} shows (i), and the proof of Theorem 7.1 in \cite{BLPSperc} shows (ii). For (iii), one has to use the fact that the existence of infinitely many ends implies $p_c<1$ (see Theorem 7.1 in \cite{BLPSperc}). Hence ${\underline {\rm gr}}\geq {\rm br}=p_c^{-1}>1$, where br is the branching number, ${\underline {\rm gr}}$ is the lower exponential growth rate, and $p_c$ is the critical percolation probability (see Sections 1.5 and 3.3 in \cite{LP} for the equality and inequality above). Part (iv) is true because otherwise there would be an isolated end in $C_o$. This is impossible by Proposition 3.9 in \cite{LS}. 
\end{proof}
\qed

\def\T{{\cal T}}

\begin{theorem} \label{comps}
Let $G$ be a unimodular quasitransitive graph. Suppose that ${\cal F}$ is an ergodic random spanning forest of $G$ that satisfies weak insertion tolerance and one of its components has infinitely many ends. Then it either has one component, or it has infinitely many components.
\end{theorem}

We mention that the proof of the same fact for Bernoulli percolation cannot be generalized to our setting directly. In case of Bernoulli percolation, one 
assumes, proving by contradiction, that there are $k$ components, $1<k<\infty$. Then one
inserts an edge, to derive that the probabilities of having $k$ components or having $k-1$ components are both positive. This contradicts ergodicity. In our case, when we apply weak insertion tolerance, even though we reduce the number of components when inserting an edge $e$, we increase it when deleting edge $f$. Hence there is no direct contradiction to the fact that the number of components is a constant a.s.

\begin{proof}
We will prove by contradiction. Suppose that there are more than one, but finitely many components. 

The proof will loosely follow the method in \cite{T}, with insertion tolerance replaced by weak insertion tolerance. Some arguments become simpler because the components are trees and also because of the assumption that there are only finitely many components. We say that two components $C$ and $C'$ {\it touch each other} at $x$ if there is an edge $\{x,y\}\in E$, $\{x,y\}\not\in {\cal F}$, with $x\in C$ and $y\in C'$.

There exist distinct components $C$ and $C'$  such that $C$ has infinitely many ends, and further, $C$ and $C'$ touch each other at infinitely many places (because the outer boundary of a cluster $C$ is infinite, and there are finitely many neighboring components). Choose $C$ and $C'$ with these properties uniformly at random, from the finitely many possible pairs. Hence there exist adjacent vertices $x$ and $y$ and an event $A_0$ such that $\P(A_0)>0$, and such that conditioned on $A_0$, $C_x=C$, $C_y=C'$. (In particular, $C_x$ has infinitely many ends, and it touches $C_y$ at infinitely many places on $A_0$.) Fix such vertices $x$ and $y$, let $e=\{x,y\}$($\in E$), and condition on $A_0$. Let the set of such touching points be $\T=\{v\in C_x \,:\,$there is a $u\in C_y$ such that $\{v,u\}\in E(G) \}$. Furthermore, for any $v\in C_x$ and infinite component $C^-$ of $C_x\setminus v$, $C^-$ has infinitely many ends (by (iv) of Lemma \ref{transient}) and $C^-\cap\T \not=\emptyset$. (This latter can be shown by a standard mass transport argument. To sketch it: one could assign to each point of $C_x$ the element of $\T$ that is closest to it in $C_x$. If the claim were not true, there would be a point that is assigned to infinitely many points of $C^-$ with positive probability, which is impossible.)

Fix $r>0$ such that given $A_0$, $C_x\setminus B(x,r)$ has at least 3 infinite components with probability at least $1/2$. (Such an $r$ exists because $C_x$ has infinitely many ends on $A_0$.) 
Let $A$ be the subevent of $A_0$ when this holds. We have $\P (A)\geq\P (A_0)/2>0$. Note that on $A$, $x\in \T$ (because this holds on $A_0$ already).

Let us sketch the rest of the proof before going into the details. We will define the following mass transport. For each $v,w$ in the same ${\cal F}$-component such that $v$ and $w$ are adjacent in $G$, take the minimal path $P_{v,w}$ within the ${\cal F}$-component between them. For each such pair $v,w$, let $v$ send mass $i^{-2}$ to the vertex of $P_{v,w}$ that has distance $i$ from $v$ in $P_{v,w}$. Then the expected mass sent out is at most $d\pi^2 /6$. However, the expected mass received is infinite, because of the way we constructed $C_x$ on $\pii A$, with an exponentially growing set of touching pairs.

Now we give the detailed proof. Let $P(C,C')\subset T$ be the set of all $v$ in $T$ with $C\setminus B(v,r)$ having at least 3 infinite components.
For the $x$ and $r$ that we fixed above, $x\in P(C,C')$ conditioned on $A$. Thus $P(C,C')\not =\emptyset$ with positive probability, and hence by ergodicity and the mass transport principle, $P(C,C')$ is infinite a.s.
Let $C_1,\ldots , C_m$ be the infinite components of $C\setminus B(x,r)$ ($m\geq 3$). 
We will show that $P(C,C')\cap C_i$  has exponential growth within $C_i$ for every $i$. Define $T(C, C')$ as the minimal subtree of $C$ that contains every vertex of $P(C,C')$. In other words, $T(C,C')$ is the union of all simple paths with both endpoints in $P(C,C')$. 
The graph $C\setminus T(C,C')$ only has finite components, as can be easily seen by a mass transport argument. (Otherwise let each vertex send mass 1 to a uniformly chosen element of $P(C,C')$ that is closest to it...)
Hence $T(C,C')\cap C_i$ is a tree with infinitely many ends (using the fact that $C_i$ has infinitely many ends), thus the growth of $T(C,C')$ is in fact exponential (Lemma \ref{transient} (iii)). Define the subtree $T_\ell (C,C')$ of $T(C,C')$ as the union of all minimal paths between two points of $P(C,C')$ such that the path has length at most $\ell$. The tree  $T_\ell (C,C')$ converges to  $T (C,C')$, and so does the expected degree within it. By (ii) in Lemma \ref{transient}, the expected degree in $T(C,C')$ is greater than 2. Hence it is greater than 2 in $T_\ell (C,C')$ as well for large enough $\ell$.
It follows that some component of $T_\ell (C,C')$ has exponential growth for $\ell$ large enough, using again Lemma \ref{transient}.
Consequently, conditioned on $A$, for large enough $\ell$ and some $c>1$, the inequality $|B_{T_\ell (C,C')}(x,r)\cap C_i \cap P(C,C')| \geq c^r$ is satisfied for each $r$ large enough. Thus $|B_{T (C,C')}(x,r)\cap C_i \cap P(C,C')| \geq c^r$ also holds. 
 
Consider the infinite components $C^1,\ldots , C^m$ of $C\setminus B(x,r)$. Conditioned on $A$, we have $m\geq 3$. On the other hand, we have seen that for each $C^i$, the set $P(C,C')\cap C^i$ has exponential growth in $C$. All but at most one of $C^1,\ldots, C^m$ are in the same component of $\pii \omega$ as $x$ ($\omega\in A$). We may assume that $C^1$ is the exceptional one (if any). 

Define the following mass transport. For every $v$ adjacent to some $w$ in $G$, choose the minimal path in $C $ between $v$ and $w$ if $v,w\in C$, and let $v$ send mass $i^{-2}$ to the $i$'th vertex on this path. The expected mass sent out is finite, because $v$ has a bounded number of neighbors. To compute the expected mass received, note that on $\pii A$, $x$ will receive mass $i^{-2}$ from every vertex of $S_{C_x}(x,i)\cap C_x \cap (P(C,C')\setminus C^1)$. Because of the exponential growth of $P(C,C')$ in $C^i$, the expected mass received is hence infinite. This contradiction finishes the proof. 

\end{proof}
\qed

\begin{remark}
{\rm One is tempted to think that the above arguments may work to show (similarly to \cite{T}) that there are no infinitely touching clusters when the percolation is weakly insertion tolerant and each component has infinitely many ends. However, this claim is not true: consider FUSF on the free product of $\Z^5$ and $\Z$, and use the result of \cite{BKPS} that any two of the infinitely many FUSF-clusters in $\Z^5$ touch each other at infinitely many places, and the fact that FUSF$(\Z^5*\Z)|_{\Z^5}$ has the same distribution as FUSF$(\Z^5)$. Weak insertion tolerance is not enough in this setting to make the argument of \cite{T} work, because deleting $f$ may make a part of the cluster ``fall off" that contains all the touching points $\T$.}
\end{remark}

\def\T{{\cal T}}

\section{Indistinguishability of clusters} \label{s.indistinguishability}

In this section we will prove the indistinguishability of clusters. By this we mean that for any  invariant measurable ${\cal A}\subset 2^{E(G)}$, a.s. either every infinite component belongs to $\A$ or none of them. When an invariant measurable ${\cal A}\subset 2^{E(G)}$ is given, we will refer to $\A$ and to $\A ^c$ 
as a {\it type}. If $C_o\in \A$, we say that $\A$ is the type of $C_o$ or (with a slight abuse of terminology) that $\A$ is the type of $o$; similarly for $\A^c$.

The following lemmas will be needed for the proof. The first one was shown in \cite{LS}. Informally speaking, it says that an invariant percolation process looks the same from a fixed vertex as from the vertex where a simple random walker is after one step within the cluster starting from the fixed vertex.

\begin{lemma}\label{stationary}
Consider an invariant edge-percolation process on $G$. Let $\hat\P_o$ be the joint distribution of $\omega$ and the two-sided delayed simple random walk on $C_o$ started from vertex $o$. Then the restriction of $\hat\P_o$ to the Aut$(G)$-invariant $\sigma$-field is stationary. More precisely, $\hat\P_o ({\cal A})=\hat\P_o ({\cal SA})$, where ${\cal S}$ is the shift operator by the random walk step, and ${\cal A}$ is any Aut$(G)$-invariant subset of $V^\Z\times 2^E$.  
\end{lemma}

\begin{definition} ({\bf Pivotal pairs}) Let $r\geq 0$ be an integer, $e,f\in E$, and let $\A$ be a type. Say that $(e,f)$, is an $r$-{\it pivotal pair}, if $f=f(\omega,e,x,r)$ for an endpoint $x$ of $e$ (as in the definition of WIT), and if the type of one of the endpoints of $e$  is different in $\pii \omega$ than in $\omega$. Define $z(e,f)=x$ if the type of $x$ is different in $\pii \omega$ than in $\omega$, otherwise define $z(e,f)=y$ (where $y$ is the other endpoint of $e$). If $(e,f)$ is an $r$-pivotal pair for some $r$, then we call $(e,f)$ {\it pivotal}.
\end{definition}

The next lemma is the modification of Lemma 3.13 in \cite{LS}, with some significant differences. Fix $r$. 

\begin{definition}\label{radon}
For each edge $\{x,y\}=e\in E$, define a measure $\P_e$ on $A_e:=\{\omega\,:\, C_x\not = C_y\}$ as $\P_e(A):=\P (\pii A)$, (where $A\subset A_e$ is arbitrary measurable and $f=f(\omega,e,x,r)$). By weak insertion tolerance, the restriction of $\P$ to $A_e$ is absolutely continuous with respect to $\P_e$. 
Let $Z_{e,x,r}(\omega)$ be the Radon-Nikodym derivative of $\P$ with respect to $\P_e$ on $A_e$. 
\end{definition}

\begin{lemma}\label{pivotal_exists}
Let $\F$ be some invariant ergodic random forest of $G$. Suppose that there exists a type $\A$ such that a.s. some cluster belongs to $\A$ and some other belongs to $\A^c$ (i.e., suppose that indistinguishability fails).
Suppose that $\F$ is weakly insertion tolerant and it has a component with infinitely many ends a.s. Then there are numbers $\delta>0$, $r\geq 0$, $p_\A >0$ such that with probability at least $p_\A$:
\begin{itemize}
\item  there exists an edge $e$ with an endpoint $x$ such that the pair $(e,f)$ is pivotal with $f=f(\omega,e,x,r)$,
\item  $Z_{e,x,r}(\omega)<\delta^{-1}$,
\item delayed simple random walk $(W(1),W(2),\ldots )$ started from $W(0)=z(e,f)$ avoids the endpoints of $e$ and $f$.
\end{itemize} 
\end{lemma}

\begin{proof}
There is a cluster with infinitely many ends, hence by Theorem \ref{ends}, all clusters have infinitely many ends.
By the assumption, with positive probability $C_x\in\A$, and $C_y\in\A^c$ for some $e=\{x,y\}\in E$. Fix such an $x$, $y$, $e$, and call the event just described as $A$. Let $A_{r,\delta}\subset A$ be the event that the following hold: $Z_{e,x,r}(\omega)<\delta^{-1}$, $C_x\setminus B(x,r)$ has at least 3 infinite components and $C_y\setminus B(x,r)$  has at least 3 infinite components. If $r$ is large enough and $\delta>0$ is small enough, then $\P (A_{r,\delta})\geq\P (A)/2$. Condition on $A_{r,\delta}$. Both $C_x$ and $C_y$ are transient by Lemma \ref{transient}. Consequently, any $f=\{u,v\}\in C_x$ with $\dist (x,f)\geq r$ is such that $C_x\setminus \{u,v\}$ has an infinite and transient component, and similarly for $C_y\setminus \{u,v\}$. Hence the last statement of the lemma holds with positive probability. What remains is to prove that $(e,f(\omega,e,x,r))$ is a pivotal pair.


Let $C_x'$ be the component of $x$ in $\pii \omega$, and $C_y$ the component of $y$ in $\omega$. Let $\B\in\{\A,\A^c\}$ be the type of $C_x'$. Since $C_y\cap C_x'$ and $C_x\cap C_x' (\not =\emptyset )$ are contained in clusters of different types in $\omega$,
either the type of the points in $C_x\cap C_x'$ changed (from $\A$ to $\A^c$ if $\B =\A^c$), or the type of the points in $C_y\cap C_x'$ changed (from $\A^c$ to $\A$ if $\B=\A$) when going from $\omega$ to $\omega '$. If the former happens with positive probability, the proof is finished with $z(e,f)=x$, otherwise with $z(e,f)=y$.
\end{proof}
\qed

\begin{theorem}\label{indist}
Let $\F $ be an invariant ergodic random forest that is weakly insertion tolerant, and such that some cluster has infinitely many ends. Then for every invariant measurable ${\cal A}\subset 2^{E(G)}$, either every infinite component belongs to ${\cal A}$ a.s., or none of them. In other words, infinite clusters are indistinguishable.  
\end{theorem}

Our proof will follow that of Theorem 3.3 in \cite{LS}, with some significant modifications. Let us mention the most important difference here. The main idea of that proof is that the existence of pivotal edges (in that setup meaning edges that connect clusters of different types) implies that by the insertion of one of them, the type of an infinite cluster changes. Hence the type of a cluster depends on the status of each of these single edges. Such edges exist arbitrarily far from the ``root" $o$ of the cluster. This makes it impossible to determine the type of the cluster of a vertex $o$ from a large enough neighborhood up to arbitrary precision, giving a contradiction. One difficulty in this sketch is that the pivotal edges are random (dependent on the configuration), hence one cannot directly apply the insertion tolerance property to bound the probabilities after insertion. This is overcome by the use of a random walk to choose the pivotal edge, at arbitrary distance, in such a way that by inserting that edge, the probability will be distorted up to some uniform factor, regardless of its distance from the root. In our setup, when only weak insertion tolerance is assumed, a further difficulty is that an edge $f$ is removed while an edge $e$ is inserted. This could, in principle, only change the type of vertices that are different from $o$. However, by proper conditioning and choosing $e$ and $f$ be far enough from each other (that is, $r$ large enough, as in the proof of Lemma\ref{pivotal_exists}), one can guarantee that infinitely many vertices change their type, including $o$, without distorting the probability of the event too much.

\begin{proof}
Fix a type $\A$ and an $r\geq 0$, $\delta>0$, $p_\A >0$ such that Lemma \ref{pivotal_exists} holds with $C_{z(e,f)}\in\A$. Fix some vertex $o$. 
Define $\A_o$ as the event that $C_o$ is of type $\A$. Given $e=\{x,y\}\in E$, $w\in\{x,y\}$ and $f=f(\omega, e,w,r)$, let $\Pef$ be the event that $(e,f)$ is $r$-pivotal with $z(e,f)=x$,  $Z_{e,x,r}(\omega)<\delta^{-1}$, $C_x\in \A$. (Note that  $Z_{e,x,r}(\omega)$ is defined whenever $(e,f)$ is pivotal.)
For an arbitrary $\epsilon >0$, let $\A_o'(\epsilon)=\A_o'$ be some fixed event that depends on only finitely many edges and satisfies $\P(\A_o\Delta \A_o')<\epsilon$. Fix $R=R(\epsilon)$ such that $\A_o'$ only depends on edges in $B(o,R)$.

Let $\W_\omega=(W_\omega(j))_{j=-\infty}^{\infty}$ be the biinfinite 2-sided delayed random walk on $C_o$ with $W(0)=o$. Define $e_n$ as a uniformly chosen edge incident to $W_\omega(n)$, let $w_n$ be a uniformly chosen endpoint of $e_n$, and let $f_n=f(\omega, e_n,w_n,r)$ if $e_n$ connects two distinct infinite components. If $e_n$ is not such (and hence the definition of $f(\omega, e_n,w_n,r)$  in WIT does not apply), then we do not define $f_n$. (We will later apply the
operator ${\pi_{e}^{f}}$ for $(e,f)=(e_n,f_n)$, but only when $(e_n,f_n)$ is pivotal.)

Let $\hat\P_o=\hat \P$ be the joint distribution of the random forest and the two-sided delayed simple random walk $\W_\omega$ on $C_o$ started from vertex $o$. 

For a fixed $e\in E$, $e=\{x,y\}$,  
denote by $\Ene (\omega)$ the event that $e_n=e$, $f_n=f$, $W_\omega (n)=x$, and that $W_\omega (j)$ is not an endpoint of $e$ or $f$ whenever $-\infty<j<n$. We mention (though we will only use this fact later) that $\Ene$ has positive probability for certain pairs $(e,f)$ and $n$ by Lemma \ref{pivotal_exists}, with the last bullett point in the lemma applied to $(W(-1),W(-2),\ldots)$.)
The WIT property implies that for any measurable $\B \subset \Pef$ such that $x$ and $y $ are in different components on $\B$,

$$\hat\P(\Ene (\omega)\cap\pii \B)=\hat\P(\Ene (\omega)|\pii\B)\P(\pii\B)=\hat\P(\Ene (\omega) | \B)\P (\pii \B)$$
$$=\frac{\P(\pii \B)}{\P (\B)}\hat\P (\Ene\cap\B) = \frac{\P(\pii \B)}{\int_{\B} Z_{e,x,r}(\omega) d\P_e}\hat\P (\Ene\cap\B). 
$$
(Recall $\P _e$ from Definition \ref{radon}.) The second equality here follows from the fact that $\Ene (\omega)$ (whose probability is coming from the random walk) does not depend on whether $e$ or $f$ is in $\omega$.
Apply the previous equality to $\B= \A_o'\cap \Pef$. Using that $Z_{e,x,r}(\omega)<\delta^{-1}$ on $\Pef$ and hence $\int_{\B} Z_{e,x,r}(\omega) d\P_e \leq \delta^{-1} \P (\pii \B)$, we obtain:

\begin{equation}\hat\P (\Ene \cap\pii (\A_o'\cap \Pef))\geq \delta \hat\P (\Ene \cap\A_o'\cap \Pef).
\label{eq:eq1}
\end{equation}

Note that if $(\omega,\W)\in \Ene\cap \Pef$, then $C_{W(n)}=C_o$ and the components of $o$ and $W(n)$ also coincide in $\piin \omega$.
We have obtained in \eqref{eq:eq1} that by closing $e$ and opening $f$ we distort the probability of our event by at most a factor of $\delta$, where ``our event" is, vaguely speaking, the event that the random walk on $C_o$ hits the endpoint $x$ of $e$ in the $n$'th step, $(e,f)$ is pivotal, and by changing the status of $e$ and $f$, the type of $x$ (and hence the type of $o$) will change. This suggests that the status of $e$ and $f$ has a high impact on the type of $o$ on this event. In what follows, we will apply this observation to all possible pivotal pairs $(e,f)$, and use the fact that the random walk hits infinitely many of them eventually. If $n$ is large enough, both $e$ and $f$ are outside of the cylinder that determines $\A_o'$, leading to a conclusion that the type of $o$ is not determined by $\A_o'$ up to a small error, a contradiction. We make this argument precise in the rest of the proof.

For $e\not\in B(o,R), f\not\in B(o,R)$,
if $(\omega,\W)\in\Ene\cap \Pefn$, then $\piin\omega\not\in \A_{W(n)}$ (by definition of $\Pef$) and thus $\piin\omega\not\in \A_o$.
Consequently, for such $e$ and $f$, $\A_o\cap \piin (\A_o'\cap  \Pefn)\cap \Ene=\emptyset$ up to measure 0. On the other hand, for $e\not\in B(o,R), f\not\in B(o,R)$, we have $\A_o'\supset \piin (\A_o'\cap  \Pefn)\cap \Ene$, because $\A_o'$ is determined by the edges in $B(o,R)$.
These observations show that $\A_o'\setminus \A_o \supset \bigcup_{e\not \in B(x,R),f\not \in B(x,R), e=\{x,y\}} \piin (\A_o'\cap  \Pefn)\cap \Ene$. This implies
the first inequality below, while the second one is by \eqref{eq:eq1}:

$$\P(\A_o'\setminus \A_o)\geq \sum_{e\not \in B(x,R),f\not \in B(x,R), e=\{x,y\}} \hat\P (\piin (\A_o'\cap \Ene \cap \Pef))$$
$$\geq \delta  \sum_{e\not \in B(x,R),f\not \in B(x,R), e=\{x,y\}} \hat\P  (\A_o'\cap \Ene \cap \Pef)\geq -\delta\epsilon+\delta  \sum_{e\not \in B(x,R),f\not \in B(x,R), e=\{x,y\}} \hat\P  (\A_o\cap \Ene \cap \Pef), $$
for every $n$.
We may assume $\delta<1$. Rewrite the inequality as
$$2\epsilon\geq\epsilon+\P(\A_o'\setminus \A_o)\geq  \delta  \sum_{e\not \in B(x,R),f\not \in B(x,R), e=\{x,y\}} \hat\P  (\A_o\cap \Ene \cap \Pef).$$
By choosing $n$ large enough, the right hand side is arbitrarily close to $\delta  \sum_{e,f\in E, e=\{x,y\}} \hat\P  (\A_o\cap \Ene \cap \Pef)$, using the transience of $C_o$ (by Lemma \ref{transient} and the definition of $\Ene$).
So fix $n(\epsilon)$ such that $\left|\delta  \sum_{e,f\in E, e=\{x,y\}} \hat\P  (\A_o\cap \Ene \cap \Pef)- \delta  \sum_{e\not \in B(x,R),f\not \in B(x,R), e=\{x,y\}} \hat\P  (\A_o\cap \Ene \cap \Pef)\right|<\epsilon$
holds for every $n\geq n(\epsilon)$. Then for $n\geq n(\epsilon)$, we have
$$3\epsilon\geq \delta  \sum_{e,f\in E, e=\{x,y\}} \hat\P  (\A_o\cap \Ene \cap \Pef)=\delta\sum_{e,f\in E, e=\{x,y\}} \hat\P  (\A_o\cap {\cal E}_{x,e,f}^0 \cap \Pef) .$$
Here the last equation holds because $\sum_{e,f\in E, e=\{x,y\}} \hat\P  (\A_o\cap {\cal E}_{x,e,f}^n \cap \Pef)$ is the same as $\sum_{e,f\in E, e=\{x,y\}} \hat\P  (\A_{W(n)}\cap {\cal E}_{x,e,f}^n \cap \Pef$), and the latter is equal to $\sum_{e,f\in E, e=\{x,y\}} \hat\P  (\A_o\cap {\cal E}_{x,e,f}^0 \cap \Pef)$
by Lemma \ref{stationary}.
The right hand side is above some uniform positive constant for every $\epsilon$ by Lemma \ref{pivotal_exists}. Letting $\epsilon$ tend to zero gives a contradiction.

\end{proof}
\qed

\begin{exmp}
The next example shows that the condition that $\F$ has a tree with infinitely many ends cannot be removed with all other conditions unchanged. That is, there exists a weakly insertion tolerant random forest $\F$ with all components infinite, but such that its components can be distinguished. Let $G':=\Z^5$ and $\F'$ be the WUSF(=FUSF) on $G'$. Let $G$ be the quasitransitive graph as follows. For each $v\in V(G')$, define two new vertices $v'$ and $v''$, and let $V(G)=\cup_{v\in V(G')} \{v,v',v''\}$. Add all edges $\{v,v'\},\{v,v''\},\{v',v''\}$ besides the edges of $G'$ (so $E(G)=\cup_{v\in V(G')} \{\{v,v'\},\{v,v''\},\{v',v''\}\}\cup E(G')$). Define $\F$ from $\F'$ by first taking $\F'$ on $E(G')\subset E(G)$. For each cluster $C$ of $\F'$, flip a coin. If it comes up head, for each $v\in C$ add one of the pairs $\{v,v'\},\{v,v''\}$ or $\{v,v'\},\{v',v''\}$ or $\{v,v''\},\{v'',v'\}$
to the edge set of $\F$, and decide which one to add uniformly, and independently over the $v$. 
If the coin came up tail, then do the same thing, but now the probability of adding edge $\{v,v'\},\{v,v''\}$ is $1/2$, while the probabilities for adding edges $\{v,v'\},\{v',v''\}$ or edges $\{v,v''\},\{v'',v'\}$ are $1/4$. This way we defined an invariant spanning forest $\F$ on $G$. Clusters containing trees of $\F'$ where the coin tosses came up head are distinguishable from those where it came up tail, from the densities of the 2-paths and ``cherries" hanging off the vertices in $V(G')$. On the other hand, using the fact that the components of $\F'$ are one-ended and that the WUSF is weakly insertion tolerant, one can check that $\F$ is also weakly insertion tolerant. (Note however that if we applied the same construction for an $\F'$ where every cluster has infinitely many ends, then the resulting $\F$ would not be weakly insertion tolerant.)
\end{exmp}

\begin{remark}\label{unimod}
All results in the paper are valid in the more general setting when $G$ is a unimodular random network. More precisely, let $(G,o)$ be an ergodic unimodular random network, as defined in \cite{AL}. The definitions of the uniform and minimal spanning forests can be extended to this setting, see Section 7 of \cite{AL}. Lemmas \ref{transient} and \ref{stationary} also hold for unimodular random graphs: the referred proofs for them are either explicitely for such graphs or generalize right away. The definition of weak insertion tolerance has to be modified by requiring the properties in Definition \ref{maindef} to hold for every edge $e$ of almost every $(G,o)$. Expectation in the proofs is understood jointly with respect to the distribution of the unimodular random graph and the random forest. To apply the proof of Theorem \ref{comps} directly, one needs to have finite expected degree for $(G,o)$. However, by using cutoff (applying mass transport only when the vertex has a degree below some properly chosen bound), one can extend the proof to an arbitrary unimodular graph.

\end{remark}

\subsection*{Acknowledgements}
I am indebted to Russ Lyons and G\'abor Pete for several useful discussions and for their comments on the manuscript.

\end{document}